\documentclass[11pt,letterpaper]{amsart}
\usepackage{amsmath,amssymb,amsthm,graphicx,mathrsfs,bbm,url}

\usepackage{amsthm}
\usepackage{wrapfig}
\usepackage{enumitem}
\usepackage[all]{xy}
\usepackage{mathtools}
\usepackage[utf8]{inputenc}

 \usepackage[T1]{fontenc}
\usepackage[usenames,dvipsnames]{color}
\usepackage[colorlinks=true,linkcolor=Red,citecolor=Green]{hyperref}
\usepackage[super]{nth}
\usepackage[open, openlevel=2, depth=3, atend]{bookmark}
\hypersetup{pdfstartview=XYZ}
\usepackage[font=footnotesize]{caption}
\usepackage{mathabx}
\usepackage[colorinlistoftodos]{todonotes}

\usepackage{epstopdf}
 
\usepackage{hyperref}

\theoremstyle{plain}
\newtheorem{theorem}{Theorem}[section]
\newtheorem*{theorem*}{Theorem}

\newtheorem{lemma}[theorem]{Lemma}

\newtheorem{proposition}[theorem]{Proposition}
\newtheorem{corollary}[theorem]{Corollary}

\theoremstyle{definition}
\newtheorem{definition}[theorem]{Definition}

\theoremstyle{remark}
\newtheorem{remark}[theorem]{Remark}

\numberwithin{equation}{section}

\newcommand{\C}{\mathbb{C}}
\newcommand{\R}{\mathbb{R}}

\newcommand{\Z}{\mathbb{Z}}

\newcommand{\V}{\mathbb{V}}

\newcommand{\HH}{\mathbb{H}}

\newcommand{\eps}{\varepsilon}

\newcommand{\mc}{\mathcal}

\newcommand{\la}{\lambda}

\newcommand{\dd}{\mathrm{d}}

\newcommand{\e}{\mathbf{e}}

\DeclareMathOperator{\WF}{WF}

\newcommand{\be}{\begin{equation}}
\newcommand{\ee}{\end{equation}}

\title
[Marked length spectrum rigidity for Anosov surfaces]
{Marked length spectrum rigidity for Anosov surfaces}

\author[Guillarmou]{Colin Guillarmou}
\address{Universit\'e Paris-Saclay, CNRS,  Laboratoire de math\'ematiques d'Orsay, 91405, Orsay, France.}
\email{colin.guillarmou@universite-paris-saclay.fr}

\author[Lefeuvre]{Thibault Lefeuvre}
\address{Université de Paris and Sorbonne Université, CNRS, IMJ-PRG, F-75006 Paris, France.}
\email{tlefeuvre@imj-prg.fr}

\author[Paternain]{Gabriel P. Paternain}
\address{Department of Pure Mathematics and Mathematical Statistics, University of
Cambridge, Cambridge CB3 0WB, UK, and Department of Mathematics, University of Washington, Seattle, WA 98195, USA.}
\email{g.p.paternain@dpmms.cam.ac.uk}

\begin{document}

\begin{abstract}
Let $\Sigma$ be a smooth closed oriented surface of genus $\geq 2$. We prove that two metrics on $\Sigma$ with the same marked length spectrum and Anosov geodesic flow are isometric via an isometry isotopic to the identity. 
The proof combines microlocal tools with the geometry of complex curves.
\end{abstract}

\maketitle

\section{Introduction}

\subsection{Main results}

On a smooth closed connected oriented manifold $\Sigma$, a metric $g$ is \emph{Anosov} if its geodesic flow on the unit tangent bundle $S\Sigma$ satisfies the Anosov property defined in \eqref{equation:anosov}. While negatively curved metrics are typical examples of Anosov metrics, the set of Anosov metrics is a considerably larger open set compared to that of negatively curved metrics (cf. \cite{Eberlein-73}). However, this set is less well understood: for instance it is still unknown whether the set of Anosov surfaces is path-connected or whether any manifold (of dimension $\geq 4$) that admits an Anosov metric also admits a negatively curved metric.  Anosov metrics coincide with the $C^2$-interior of metrics without conjugate points \cite{Ruggiero-91}, and they can also exist isometrically embedded in $\mathbb{R}^3$ \cite{Donnay-Pugh-03}, unlike negatively curved metrics.

On $\Sigma$, we shall denote by $\mc{M}_{\mathrm{Anosov}}(\Sigma)$ the set of all Anosov metrics (assuming it is non-empty) and by $\mathbf{M}_{\mathrm{Anosov}}(\Sigma) := \mc{M}_{\mathrm{Anosov}}(\Sigma)/\mathrm{Diffeo}^0(\Sigma)$ the \emph{moduli space} of Anosov metrics, where we denote by $\mathrm{Diffeo}^0(\Sigma)$ the group of diffeomorphisms isotopic to the identity.
Let $\mc{C}$ be the set of free homotopy classes on $\Sigma$. This set is countable and in natural correspondence with conjugacy classes of $\pi_1(\Sigma,\star)$. It is well-known that for an Anosov metric $g$ there exists a unique closed geodesic $\gamma_g(c)$ in each free homotopy classes $c \in \mc{C}$ (see \cite{Klingenberg-74} or \cite[Theorem 29]{CIPP00} for a proof that covers Anosov energy levels of Lagrangian systems). The marked length spectrum is then defined as the map
\begin{equation}
\label{equation:mls}
\mc{L} : \mathbf{M}_{\mathrm{Anosov}}(\Sigma) \to (0,\infty)^{\mc{C}}, \qquad \mc{L}_g(c) := \ell_g(\gamma_g(c)),
\end{equation}
where $\ell_g(\gamma)$ denotes the length of the curve $\gamma$ computed with respect to $g$.

It is conjectured that the marked length spectrum map \eqref{equation:mls} is injective. For negatively curved metrics, this is known as the Burns-Katok conjecture \cite{Burns-Katok-85}. The purpose of the present paper is to establish this conjecture when $\dim \Sigma =2$, that is, when $\Sigma$ is a surface.

\begin{theorem}
\label{theorem:main}
Let $\Sigma$ be a smooth closed connected oriented surface. Then the marked length spectrum $\mc{L} : \mathbf{M}_{\mathrm{Anosov}}(\Sigma) \to (0,\infty)^{\mc{C}}$ is injective. In other words, if $g_1$ and $g_2$ are two Anosov metrics on $\Sigma$ with same marked length spectrum (that is, $\mc{L}_{g_1} = \mc{L}_{g_2}$), then there exists a smooth diffeomorphism $\phi : \Sigma \to \Sigma$, isotopic to the identity, such that $\phi^* g_1 = g_2$.
\end{theorem}
The result also applies to metrics with $C^4$-regularity, see Remark \ref{remark:regularity}.

The theorem can be seen as the closed surface analogue of Pestov and Uhlmann's celebrated boundary rigidity result for simple surfaces \cite{Pestov-Uhlmann-05}. A surface is simple if it has strictly convex boundary, is non-trapping, and has no conjugate points (like the Anosov property, this is a $C^2$-open condition in the metric). In this case, there is a unique geodesic joining any two boundary points, similar to the boundary at infinity of the universal cover of an Anosov surface (which is a disk). The boundary rigidity problem seeks to determine the metric (up to an isometry that is identity on the boundary) from the boundary distance function.

Equality of boundary distance functions for two simple metrics induces a conjugacy between their geodesic flows fixing the boundary of the unit tangent bundle. Similarly, equality of marked length spectra for Anosov surfaces induces a conjugacy between their geodesic flows homotopic to the identity (i.e. fixing the boundary at infinity of the universal cover). A potential approach to both problems is to show that the conjugacy is the lift of an isometry (up to an innocuous action of one of the two flows). Croke \cite{Croke-90} and Otal \cite{Otal:1990ko,Otal-90} used this approach to prove both boundary rigidity and marked length spectrum rigidity for non-positively curved surfaces. In contrast, Pestov and Uhlmann \cite{Pestov-Uhlmann-05} showed that the scattering relation determines the boundary values of holomorphic functions, and relied on the solution of the Calder\'on problem to prove boundary rigidity for all simple metrics. 
The problem of marked length spectrum rigidity for Anosov surfaces in the closed setting has remained open until the present work.
A {\it posteriori}, we will also prove that conjugacies must have the expected form, see Corollary \ref{corollary} below.

Our proof draws inspiration from Pestov-Uhlmann's approach that depends on utilizing ``fiberwise holomorphic'' smooth invariant functions for the geodesic flow to recover the conformal structure of the surface (\textcolor{black}{after a suitable reinterpretation of their original proof}). However, in the case of closed surfaces, there is no natural Calderón problem connected to the marked length spectrum as in \cite{Pestov-Uhlmann-05}, and the smooth invariant functions need to be replaced by suitable singular invariant distributions, which has made it difficult to apply this strategy in our context until now. We take a different route and use the period matrix to recover the conformal structure via the Torelli theorem, see the strategy outline below.
{\color{black}The boundary rigidity for simple surfaces in fact follows from Theorem \ref{theorem:main} in conjunction with the embedding theorem \cite{Chen-Erchenko-Gogolev-20}; this was established by Erchenko and the second author in \cite{Erchenko-Lefeuvre-24}.}

As mentioned above, prior to Theorem \ref{theorem:main}, the only known cases of injectivity of the marked length spectrum for closed surfaces were essentially limited to metrics with non-positive curvature (see \cite{Croke-90,Otal-90} and the subsequent generalization by Croke, Fathi, and Feldman \cite{Croke-Fathi-Feldman-92}). These proofs {\it do not} extend to the Anosov setting as they rely crucially on the assumption that the Gauss curvature is non-positive (or that the Morse correspondence preserves angles); thus a new approach is needed to establish Theorem \ref{theorem:main}.
 For dimensions greater than or equal to three, fewer results exist. Given two Anosov metrics in the same conformal class, Katok's argument in \cite{Katok-88} provides injectivity of the marked length spectrum. Almost all other results require non-positive sectional curvature: Hamenstädt \cite{Hamenstadt-99} proved the conjecture when one of the two metrics is locally symmetric, relying on the entropy paper of Besson-Courtois-Gallot \cite{Besson-Courtois-Gallot-95}. Recently, the first two authors of this paper proved the conjecture in any dimension when the curvature is non-positive \cite{Guillarmou-Lefeuvre-19} and 
 the metrics are assumed to be close. An alternative proof based on the concept of geodesic stretch can be found in \cite{Guillarmou-Knieper-Lefeuvre-22}. Finally, in the Anosov setting and without requiring curvature assumption, the second author and Cekić proved in \cite[Corollary 1.4]{Cekic-Lefeuvre-21-2} that the marked length spectrum is locally injective near a generic Anosov metric in any dimension.
 
 As a historical remark, we note that the easier question of {\it infinitesimal} marked length spectrum rigidity was first resolved for negatively curved surfaces in a landmark paper by Guillemin and Kazdhan \cite{Guillemin-Kazhdan-80}. This paper has had a profound influence in the development of geometric inverse problems in two dimensions in the last two decades, see \cite{Paternain-Salo-Uhlmann-book}. We also note that {\it infinitesimal} marked length spectrum rigidity for Anosov surfaces was established in \cite{Paternain-Salo-Uhlmann-14-2} and this relies in turn on the possibility of lifting an arbitrary holomorphic 1-form to an invariant distribution; an important result, also used here.
 
As a direct consequence of Theorem \ref{theorem:main}, the proof of \cite[Theorem 1.1]{Croke-Dairbekov-04} provides a more general version of Theorem \ref{theorem:main} related to the minimal filling problem of Gromov, see \cite[Question 6.8]{Croke-04} where this is discussed.

\begin{corollary}
\label{theorem:main2}
Let $\Sigma$ be a smooth closed connected oriented surface and $g_1, g_2 \in \mathbf{M}_{\mathrm{Anosov}}(\Sigma)$. If $\mc{L}_{g_1} \geq \mc{L}_{g_2}$, then $\mathrm{vol}_{g_1}(\Sigma) \geq \mathrm{vol}_{g_2}(\Sigma)$ with equality of volumes if and only if there exists a smooth diffeomorphism $\phi : \Sigma \to \Sigma$, isotopic to the identity, such that $\phi^* g_1 = g_2$.
\end{corollary}

In Corollary \ref{theorem:main2}, the volumes of the metrics are computed with respect to the Riemannian measure. Corollary \ref{theorem:main2} follows from the positive Liv\v sic Theorem of Lopes and Thieullen \cite{Lopes-Thieullen-05} and Theorem \ref{theorem:main}, see \cite{Croke-Dairbekov-04} for a proof. Finally, Theorem \ref{theorem:main} also implies a classification result for smooth conjugacies of Anosov geodesic flows over surfaces:

\begin{corollary}
\label{corollary}
Let $\Sigma$ be a smooth closed connected oriented surface, let $g_1,g_2 \in \mathbf{M}_{\mathrm{Anosov}}(\Sigma)$ and denote by $(\varphi^{g_i}_t)_{t \in \R}$ the geodesic flow over the unit tangent bundle $S\Sigma_i$ of $(\Sigma,g_i)$ for $i=1,2$. If $\phi : S\Sigma_1 \to S\Sigma_2$ is a diffeomorphism such that
\[
\phi \circ \varphi^{g_1}_t = \varphi^{g_2}_t \circ \phi, \qquad \forall t \in \R,
\]
then there exists a time $t_0 \in \R$ and an isometry $F : (\Sigma,g_1) \to (\Sigma,g_2)$ such that $\phi = \varphi^{g_2}_{t_0} \circ F'$, where $F' : S\Sigma_1 \to S\Sigma_2$ is defined as $F'(x,v) := (F(x), \dd F_x(v))$.
\end{corollary}

Taking $g_1 = g_2$ in Corollary \ref{corollary}, we thus find that the set of self-conjugacies of an Anosov geodesic flow over a surface is given (after taking the quotient by the flow) by the finite group of isometries of the metric. In the case of non-positively curved surfaces, Corollary \ref{corollary} was first obtained in \cite[Theorem B]{Croke-90}.  In \cite[Theorem C]{Croke-Fathi-Feldman-92} an argument is given that shows that in fact the claim in  Corollary \ref{corollary} follows from the marked length spectrum rigidity. The only point to note is that \cite[Lemma 4.3]{Croke-Fathi-Feldman-92} holds for an Anosov metric: indeed, a self-conjugacy isotopic to the identity must map a closed geodesic to itself (since there is a unique closed geodesic in each free homotopy class) and by density of closed orbits and transitivity it must be of the form $\varphi_{t_{0}}$.
In dimensions three and higher, finiteness of the group of self-conjugacies (after modding out by the flow itself) is known in some cases such as $1/4$-pinched negatively curved manifolds, we refer to \cite{Damjanovic-Wilkinson-Disheng-21} for further details.

\subsection{Strategy}
\label{ssection:strategy}

The proof strategy involves demonstrating that the marked length spectrum $\mc{L}_g$ captures the complex structure of the metric $[g]$ up to biholomorphisms isotopic to the identity, and hence determines the class of the underlying complex structure in Teichmüller space (Proposition \ref{proposition:key2}). Once this is established, the injectivity of the marked length spectrum in the same conformal class, as shown by Katok \cite{Katok-88}, leads to a straightforward proof of Theorem \ref{theorem:main}.

To establish that $\mc{L}_g$ determines the complex structure, we show that it encodes the period matrix of the underlying Riemann surface (Proposition \ref{proposition:key}) and then observe that the argument may be repeated on any finite cover
to recover the structure in Teichmüller space. This relies on the fact that given two different points
in Teichm\"uller space, there is a finite cover of the surface where the lifted complex structures are in different orbits of the mapping class group
(Lemma \ref{lemma:cover}).

\textcolor{black}{To recover the period matrix, we first show that holomorphic 1-forms on the Riemann surface associated to an Anosov metric correspond to the first Fourier modes of a certain space  $\mc{A}_+(S\Sigma)$ of fiberwise holomorphic flow-invariant distributions. These are distributions on $S\Sigma$ that are invariant under the geodesic flow and have a specific Fourier decomposition in the circle fibers of $S\Sigma$ (see \S\ref{sssection:degree}).
Their first Fourier mode corresponds to a genuine holomorphic 1-form on the Riemann surface. 
Conversely, an arbitrary holomorphic 1-form on the surface can be seen as the first Fourier mode of an element in $\mc{A}_+(S\Sigma)$.
This is a non-trivial fact and relies on advances in tensor tomography on surfaces \cite{Paternain-Salo-Uhlmann-14-2,Guillarmou-17-1}.}

\textcolor{black}{Finally, we show in Lemma \ref{lemma:magical} that the integral of any holomorphic $1$-form $f$ along a closed oriented geodesic 
$\gamma$ can be expressed, up to a multiplicative constant, as the integral of $u\, \dd\la$ on $S^1\gamma=\pi^{-1}(\gamma)$ (where $\pi:S\Sigma\to \Sigma$ is the projection and $\dd\la$ the Liouville $2$-form) for any $u\in \mc{A}_+(S\Sigma)$ having $f$ as its first non-zero Fourier mode. 
This allows for the recovery of integrals of holomorphic $1$-forms on closed geodesics, and thus the period matrix, from the conjugacy class of the flow (Proposition \ref{proposition:period}).
We note that the pairing formula of Lemma \ref{lemma:magical} shares some similarity with the intersection number for currents \emph{à la} Bonahon \cite{Bonahon-88}, Otal \cite{Otal-90}, and others.}  

\subsection{Organization of the paper}

The proof of Theorem \ref{theorem:main} relies on several tools, which we introduce in Section \ref{section:tools}. {\color{black}Specifically, \S\ref{ssection:complex} provides a brief review of the geometry of complex curves, while \S\ref{ssection:unit} is devoted to the geometry and harmonic analysis of the unit tangent bundle of a surface. In \S\ref{sssection:degree}, we introduce the notion of fiberwise holomorphic invariant distributions, and connect them to holomorphic $1$-forms on the surface.
 In \S\ref{ssection:hyperbolic}, we delve into hyperbolic dynamics and tensor tomography; we also compute the wavefront set of fiberwise holomorphic invariant distributions. The last paragraph \S\ref{ssection:pairing} is concerned with establishing a signficant pairing formula.} Finally, we present the proof of Theorem \ref{theorem:main} in Section \ref{section:proof}.\\
 
\noindent \textbf{Acknowledgements.} We warmly thank J. Marché, B. Petri, I. Smith and M. Wolff for very helpful discussions related to the proof of Lemma \ref{lemma:cover}. We also thank C. Matheus for pointing out an error in an earlier draft and the referees for several comments and suggestions for improvement.

\section{Background and preliminary lemmas}

\label{section:tools}

\subsection{Complex geometry}

\label{ssection:complex}

In what follows, $\Sigma$ is a smooth closed oriented surface of genus $\geq 2$.

\subsubsection{General facts} Let $g$ be a smooth Riemannian metric on $\Sigma$. The conformal class $[g]$ of $g$ (and the orientation of $\Sigma$) induces a complex structure $J \in C^\infty(\Sigma,\mathrm{End}(T\Sigma))$ on $\Sigma$, turning it into a Riemann surface which we shall denote by $(\Sigma,J)$. 

We denote by $\mc{T}(\Sigma)$ the Teichmüller space of $\Sigma$, that is, the space of complex structures $J$ on $\Sigma$ modulo the equivalence relation that $J \sim J'$ iff there exists a diffeomorphism $\psi : \Sigma \to \Sigma$, isotopic to the identity, such that $\psi^*J = J'$. Such an equivalence class of complex structures will be denoted by $[J]$. The mapping class group $\mathrm{MCG}(\Sigma)$ is defined as the quotient of orientation preserving diffeomorphisms $\mathrm{Diff}^+(\Sigma)$ modulo isotopy.

There is a well-defined action of $\mathrm{MCG}(\Sigma)$ on $\mc{T}(\Sigma)$ (by pullback). The quotient space $\mc{M}(\Sigma) := \mc{T}(\Sigma)/\mathrm{MCG}(\Sigma)$ is called the \emph{moduli space} of (complex structures on) $\Sigma$. We refer to \cite{Farb-Margalit-11} for a general introduction to the mapping class group of closed surfaces.

\subsubsection{Period matrix} Let $\left\{a_i,b_j\right\}$ be a canonical basis of the homology group $H_1(\Sigma,\Z)$ on the surface $\Sigma$. Let $(\Sigma,J)$ be a Riemann surface structure on $\Sigma$ and denote by $H_J^0(\Sigma,K^m)$ the space of holomorphic sections of the $m$-th power of the canonical bundle $K := T^*_{\C}\Sigma^{1,0}$ of $(\Sigma,J)$, for $m\geq 1$.
It is well-known that there exists a unique basis $\left\{\zeta_i\right\}$ of holomorphic 1-forms in $H_J^0(\Sigma,K)$ such that
\begin{equation}
\label{equation:delta}
\int_{a_j} \zeta_k = \delta_{jk},
\end{equation}
see \cite[Proposition, page 63]{Farkas-Kra-92} for instance. The \emph{period matrix} of $(\Sigma,J)$ is then defined as the matrix $\Pi(J)$ whose $jk$-entry is
\[
\Pi(J)_{jk} = \int_{b_j} \zeta_k.
\]
It is a symmetric matrix with positive definite imaginary part. The space of symmetric matrices with positive definite imaginary part and size given by the genus of $\Sigma$ is called the \emph{Siegel upper half-space} $\mc{H}(\Sigma)$. Hence, we get a well-defined period matrix map
\[
\Pi : \mc{T}(\Sigma) \longrightarrow \mc{H}(\Sigma).
\]
We will need the Torelli theorem in the following form: 
\begin{theorem}
\label{theorem:torelli}
Assume that $\Sigma$ has genus $\geq 2$. If $\Pi(J_1) = \Pi(J_2)$, then there exists a diffeomorphism $\psi : \Sigma \to \Sigma$ such that $\psi^*J_2 =  J_1$.
\end{theorem}

We refer to \cite[Theorem III.12.3]{Farkas-Kra-92} or \cite[p. 359]{Griffiths-Harris-94} for a proof. Actually, it can be proved that $[\psi] \in \mathrm{MCG}(\Sigma)$ lives in a degree $2$ extension of the \emph{Torelli group} that is, it acts as $\pm \mathbf{1}$ on homology $H_1(\Sigma,\Z)$ but this will not be needed in what follows.

\subsection{Fourier analysis on the unit tangent bundle}

\label{ssection:unit}

\textcolor{black}{We aim to recover the period matrix, that is, the integrals of holomorphic $1$-forms from the conjugacy class of the flow. There is a particularly useful relation between the Fourier decomposition of functions (or distributions) in the fibers of the unit tangent bundle and sections of powers of the canonical line bundle over $\Sigma$, as we now explain following \cite[Section 3]{Guillemin-Kazhdan-80}.
}
Let
\[
S\Sigma := \left\{ (x,v) \in T\Sigma ~|~ |v|_g=1\right\}
\]
be the unit tangent bundle of $(\Sigma,g)$ and $\pi:S\Sigma\to \Sigma$ the projection. 

Let $(\varphi_t)_{t \in \R}$ be the geodesic flow on $S\Sigma$ and $X$ its infinitesimal generator. Let $V$ be the vertical vector field generating the $\mathrm{SO}(2)$-rotation group $(R_{\theta})_{\theta \in [0,2\pi]}$ in the fibers and let $\V := \R V$. \textcolor{black}{Notice that for two conformally related metrics $g'=e^{\omega}g$, the associated vertical vector fields $V$ and $V'$ agree after identification of
$S_g\Sigma$ with $S_{g'}\Sigma$ by scaling.}
Define $H := -[X,V]$ and $\HH := \R H$. The vector fields $\left\{X, H, V\right\}$ form an orthonormal basis on $T(S\Sigma)$ for the \emph{Sasaki metric} (the natural lift of $g$ to $S\Sigma$). 

The \emph{Liouville} $1$-form $\lambda \in C^\infty(S\Sigma, T^*(S\Sigma))$ is defined by $\lambda(X) = 1$ and $\lambda(H)=\lambda(V)=0$. It is invariant by the geodesic flow, that is, $\mc{L}_X \lambda = 0$. Moreover, $\dd\lambda$ is a $2$-form that is non-degenerate on the \emph{contact plane} $\R H \oplus \R V$ and such that $\iota_X \dd\lambda = 0$. Hence
\[
\mu :=-\lambda \wedge \dd\lambda
\]
is a volume-form, invariant by the geodesic flow, called the \emph{Liouville} volume form. Equivalently, $\mu$ is the Riemannian volume form induced by the Sasaki metric on $S\Sigma$. From now on, the $L^2$ space on $S\Sigma$ is defined as $L^2(S\Sigma) := L^2(S\Sigma, \mu)$.

We define the $1$-forms $\beta, \psi$ on $S\Sigma$ by $\beta(H)=1=\psi(V)$ and $\beta(X)=\beta(V)=0=\psi(X)=\psi(H)$. It can then be checked that
\begin{equation}
\label{equation:dlambda}
\dd\lambda =  \psi \wedge \beta , \qquad \mu = \lambda \wedge \beta \wedge \psi.
\end{equation}
We set $(E^0)^* := \R\lambda$, $\HH^* := \R \beta$ and $\V^* := \R \psi$. We refer to \cite{Paternain-99} and \cite[Chapter 3]{Paternain-Salo-Uhlmann-book} for further details on the geometric structure on $S\Sigma$, see also Figure \ref{figure} below for a representation of the bundles introduced above.

We introduce the complex line bundle $\Omega_1 \to \Sigma$ whose fiber over $x \in \Sigma$ is given by:
\[
(\Omega_1)_x := \left\{ u(x, \cdot) ~|~ u \in C^\infty(S\Sigma), Vu = iu\right\}.
\]
The line bundle $\Omega_1$ is isomorphic to the canonical line bundle $K$ of the underlying Riemann surface $(\Sigma,J)$, that is, there exists a fiberwise linear map $\pi_1^* : K \to \Omega_1$ given for all $x \in \Sigma$ by
\[
\pi_1^* : K_x \ni f \mapsto \left(S_x\Sigma \ni v \mapsto f(v)\right) \in (\Omega_1)_x.
\]
The powers $\Omega_n := \Omega_1^{\otimes n}$ for $n \in \Z$, correspond to
\[
(\Omega_n)_x = \left\{ u(x, \cdot) ~|~ u \in C^\infty(S\Sigma), Vu = inu\right\}
\]
and
\begin{equation}
\label{equation:pin-star}
\pi_n^* : K_x^{\otimes n} \ni f \mapsto  \left(S_x\Sigma \ni v \mapsto f(v,...,v)\right) \in (\Omega_n)_x
\end{equation}
is an isomorphism. Hence, from now on, we will freely identify $K^{\otimes n}$ with $\Omega_n$ via $\pi_n^*$. We denote its adjoint by ${\pi_n}_*$ (with respect to {\color{black}the natural inner products on $L^2(\Sigma,K^{\otimes n})$ and $L^2(\Sigma,\Omega_n)$}). 

{\color{black}Notice that, by definition, $C^\infty(\Sigma,\Omega_n) \subset C^\infty(S\Sigma)$ is naturally a subset of $L^2$-functions on $S\Sigma$. Hence $\pi_n^* : C^\infty(\Sigma,K^{\otimes n}) \to C^\infty(S\Sigma)$ and its adjoint ${\pi_n}_*$ can also be seen as a map ${\pi_n}_* : C^\infty(S\Sigma) \to  C^\infty(\Sigma,K^{\otimes n})$. It is straightforward to check that ${\pi_n}_*(C^\infty(\Sigma,\Omega_{n'})) = 0$ for $n \neq n'$.}

The map $(2\pi)^{-1} {\pi_n}_* \pi_n^*$ is the identity on $C^\infty(S\Sigma,K^{\otimes n})$ and
\begin{equation}
\label{equation:proj}
\dfrac{1}{2\pi} {\pi_n}^* {\pi_n}_* : L^2(S\Sigma) \to L^2(S\Sigma)
\end{equation}
is the $L^2$-orthogonal projection onto the $n$-th Fourier mode. Any function $f \in L^2(S\Sigma)$ can thus be decomposed as $f = \sum_{n \in \Z} f_n$, where $f_n := (2\pi)^{-1} {\pi_n}^* {\pi_n}_* f$.

\subsection{Fiberwise holomorphic distributions.}

\label{sssection:degree}

\textcolor{black}{We now relate holomorphic $1$-forms on $\Sigma$ to the geodesic flow on $S\Sigma$. This connection is made by considering a class of distributions on $S\Sigma$, called fiberwise holomorphic invariant distributions, with the property that they are invariant by the geodesic flow and only have non negative Fourier modes.
} 

Let $\mc{D}'(S\Sigma)$ denote the space of distributions on $S\Sigma$ seen as the topological dual of volume forms, that is:
\[
\mc{D}'(S\Sigma) := (C^\infty(S\Sigma, \Lambda^3 T^* S\Sigma))'.
\]
Elements in $\mc{D}'(S\Sigma)$ are ``generalized'' functions which can be paired intrinsically against smooth volume forms. The \emph{wavefront set} $\WF(f)$ of a distribution $f \in \mc{D}'(S\Sigma)$ describes the (co)directions in $T^*(S\Sigma)$ in which the distribution is irregular, see \cite[Chapter 8]{Hormander-90} for a detailed account. 

The vector field $X$ acts on $\mc{D}'(S\Sigma)$ by duality, namely, given $u \in \mc{D}'(S\Sigma)$ and $\omega \in C^\infty(S\Sigma, \Lambda^3 T^* S\Sigma)$, $(Xu,\omega) := -(u,\mc{L}_X \omega)$. Just as for $L^2$-functions, any $f \in \mc{D}'(S\Sigma)$ can be decomposed as a sum of Fourier modes
\begin{equation}
\label{equation:decomp}
f = \sum_{n \in \Z} f_n,
\end{equation}
where $f_n \in \mc{D}'(S\Sigma)$ satisfies $V f_n = in f_n$ and $f_n := \tfrac{1}{2\pi}{\pi_n}^* {\pi_n}_* f$. A distribution (or a function) is said to have \emph{finite Fourier degree} if the sum \eqref{equation:decomp} only contains a finite number of terms.

\textcolor{black}{
\begin{definition}
 A distribution (or a function) is \emph{fiberwise holomorphic} if \eqref{equation:decomp} only contains non negative Fourier modes. 
\end{definition}
Defining the Szegö projections $\mathbf{S}_\pm : \mc{D}'(S\Sigma) \to \mc{D}'(S\Sigma)$ by
\begin{equation}
\label{equation:szego}
\mathbf{S}_+\left(\sum_{n \in \Z} f_n\right) :=\sum_{n \geq 0} f_n, \qquad \mathbf{S}_-\left(\sum_{n \in \Z} f_n\right) :=\sum_{n \leq 0} f_n.
\end{equation}
Notice that a distribution $f$ is fiberwise holomorphic iff $\mathbf{S}_+f =f$.}


The geodesic vector field $X$ splits into $X=\eta_++\eta_-$ where 
\begin{equation}
\label{equation:etapm}
\eta_\pm := \frac{1}{2}(X \mp iH).
\end{equation}
These operators are called the \emph{raising} ($+$) and \emph{lowering} ($-$) {\color{black}Guillemin-Kazhdan} operators: 
 $\eta_\pm$ act as raising/lowering operators on the Fourier decomposition \eqref{equation:decomp}, that is:
\[
\eta_\pm : C^\infty(\Sigma, \Omega_n) \to C^\infty(\Sigma,\Omega_{n\pm 1})
\]
is continuous.

{\color{black}For $n \geq 0$, the isomorphism $\pi_n^*$ defined in \eqref{equation:pin-star} intertwines the operator $\eta_-$ acting on $C^\infty(\Sigma,\Omega_n)$ and the $\bar{\partial}$ operator on $C^\infty(\Sigma,K^{\otimes n})$, see \cite[Lemma 2.1]{Paternain-Salo-Uhlmann-14-2} and its proof. In addition, the dimension of $\ker \eta_\pm|_{C^\infty(\Sigma,\Omega_n)}$ can then be expressed in terms of the genus of $\Sigma$ by the Riemann-Roch theorem.}

Define
\begin{equation}
\label{equation:h0}
H^0(\Sigma,\Omega_n) := \left\{ u \in C^\infty(S\Sigma)\,|\,  Vu = inu, \;\eta_- u = 0\right\}.
\end{equation}
We shall denote by $H^0_J(\Sigma,K^{\otimes n})$ the complex vector space of holomorphic differentials of degree $n$. The subscript $J$ indicates that this is computed with respect to the complex structure $J$. Observe that by the previous discussion, $\pi_n^*$ identifies $H^0_J(\Sigma,K^{\otimes n})$ with $H^0(\Sigma,\Omega_n)$ for $n \geq 0$.
  
 Note that the operator $X = \eta_- + \eta_+$ acts on the decomposition \eqref{equation:decomp} as
\[
X\left(\sum_{n \in \Z} f_n\right) = \sum_{n \in \Z} \eta_+ f_{n-1} + \eta_- f_{n+1}.
\] 
\textcolor{black}{This decomposition goes back to Guillemin-Kazhdan \cite{Guillemin-Kazhdan-80}.
}
Let
\begin{equation}\label{DefA_+}
\mc{A}_+(S\Sigma) := \left\{ f \in \mc{D}'(S\Sigma) ~|~ X f = 0, \; \mathbf{S}_+ f = f\right\},
\end{equation}
be the {\color{black}set} of flow-invariant fiberwise holomorphic distributions.

\begin{remark}It can be proved that $\mc{A}_+(S\Sigma)$ is an algebra, regardless of the underlying Riemannian geometry of the surface, that is, multiplication is well-defined and continuous with respect to the $\mc{D}'(S\Sigma)$ topology, see \cite[Theorem 1.1]{Bohr-Lefeuvre-Paternain-23}; this property was crucial in \cite{Paternain-Salo-Uhlmann-14-2,Guillarmou-17-1} for proving injectivity of the X-ray transform on symmetric tensors of order $\geq 2$ on Anosov surfaces (though the algebra property will not be needed in the proof of Theorem \ref{theorem:main}).
\end{remark}

\begin{lemma}\label{lem:A+_holof_0f_1} 
Given $f\in \mc{A}_+(S\Sigma)$, then $f_0$ is constant and $\eta_{-}f_{1}=0$, so $f_{1}$ is naturally identified with a holomorphic $1$-form.
\end{lemma}

\begin{proof} Note that $Xf=0$, implies in particular that the Fourier modes $(Xf)_{-1}$ and $(Xf)_{0}=0$ must vanish. Using that $X=\eta_{+}+\eta_{-}$ and the fact that $f$ is fibrewise holomorphic, we see that that $\eta_{-}f_0=0$ and $\eta_{-}f_{1}=0$. Since holomorphic functions on $\Sigma$ are constant the result follows.
\end{proof}

\subsection{Hyperbolic dynamics and the wavefront set of elements in $\mathcal A_{+}(S\Sigma)$}

\label{ssection:hyperbolic}

We now further assume that $(\Sigma,g)$ is an \emph{Anosov} metric, that is, the geodesic flow on $S\Sigma$ is Anosov (or \emph{uniformly hyperbolic}).

\subsubsection{Definition. First properties}

Recall that the geodesic flow $(\varphi_t)_{t \in \R}$ is \emph{Anosov} if there exists a flow-invariant continuous splitting 
\[
T(S\Sigma) = \R X \oplus E^s \oplus E^u
\]
and uniform constants $C,\lambda > 0$ such that
\begin{equation}
\label{equation:anosov}
\begin{array}{ll}
 |d\varphi_t(w)| \leq C e^{-\lambda t}|w|, & \qquad  \forall t \geq 0, \forall w \in E^s, \\
  |d\varphi_{-t}(w)| \leq C e^{-\lambda t}|w|, & \qquad  \forall t \geq 0, \forall w \in E^u.
\end{array}
\end{equation}
By \cite{Klingenberg-74} (see also \cite{Mane-87} for an alternative proof), the bundles $E^s$ and $E^u$ are known to intersect the vertical bundle $\V$ trivially, that is
\begin{equation}
\label{equation:disjoint0}
E^s \cap \V = E^u \cap \V = \{0\}.
\end{equation}
\textcolor{black}{Moreover, there exist functions $r_\pm\in \bigcap_{\eps>0}C^{2-\eps}(S\Sigma)$ }solutions to the Riccati equations
\[
X r_\pm + r_\pm^2 + \pi^*\kappa = 0,
\]
where $\kappa$ denotes the Gauss curvature on $(\Sigma,g)$, such that
\[
E^s = \R(H + r_- V), \qquad E^u = \R(H+r_+ V).
\]
If $\kappa < 0$, then $r_- < 0$ while $r_+ > 0$. We set $Y^s := H + r_- V$ and  $Y^u := H+r_+ V$. Note that the basis $\left\{X,Y^s,Y^u\right\}$ is positively-oriented.

\begin{center}
\begin{figure}[htbp!]
\includegraphics[scale=0.4]{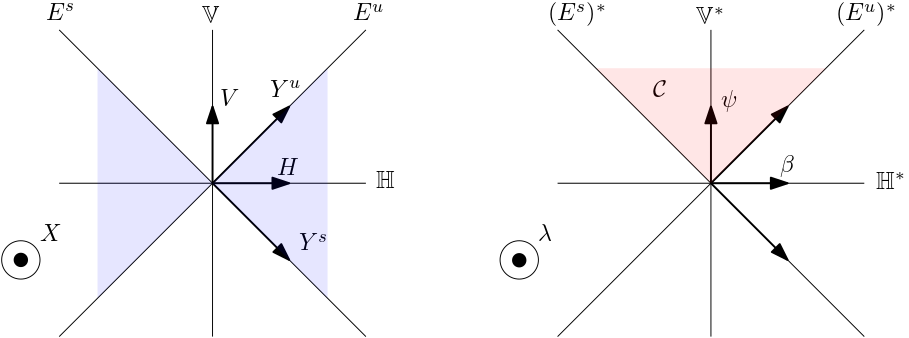}
\caption{The various subbundles in tangent and cotangent space.}
\label{figure}
\end{figure}
\end{center}

The dual bundles are defined by
\[
(E^0)^*(E^s \oplus E^u) = 0 = (E^s)^*(E^0\oplus E^s) = (E^u)^*(E^0 \oplus E^u),
\]
where $E^{0}:= \R X$. It can be checked that $(E^s)^*$ (resp. $(E^u)^*$) satisfies similar estimates \eqref{equation:anosov} to $E^s$ (resp. $E^u$) with $d\varphi_t$ being replaced by $d\varphi_t^{-\top}$ (inverse transpose). Equation \eqref{equation:disjoint0} becomes in the cotangent bundle
\begin{equation}
\label{equation:disjoint}
(E^s)^* \cap \HH^* = (E^u)^* \cap \HH^* = \{0\}.
\end{equation}

\subsubsection{Wavefront set properties}

We introduce 
\[
\mc{C} \subset \{(v,\xi)\in T^*(S\Sigma) \,|\, \xi(X(v))=0\},
\]
the closed cone enclosed by $(E^s)^*$ and $(E^u)^*$ in the half-space $\left\{ \xi(V(v)) \geq 0\right\}$, see Figure \ref{figure} for a description.

\begin{lemma}
\label{lemma:wf-c}
Let $f \in \mc{A}_+(S\Sigma)$. Then $\WF(f) \subset \mc{C}$ and ${\pi_n}_* f \in C^\infty(\Sigma,K^{\otimes n})$ for all $n \in \Z$.
\end{lemma}

\begin{proof}
Observe that $f \in \mc{A}_+(S\Sigma)$ is fiberwise holomorphic, i.e. $\mathbf{S}_+f=f$ so
\begin{equation}
\label{equation:wf}
\WF(f) = \WF(\mathbf{S}_+f) \subset \left\{(v,\xi) \in T^*(S\Sigma) ~|~ \xi(V(v)) \geq 0 \right\},
\end{equation}
by using the wavefront set description of the Schwartz kernel of $\mathbf{S}_+$, see \cite[Lemma 3.10]{Guillarmou-17-1}. Moreover, $X f = 0$ so by elliptic regularity one has
\[
\WF(f)\subset \left\{(v,\xi) \in T^*(S\Sigma) ~|~ \xi(X(v)) =0 \right\},
\]
while by standard propagation of singularities (see \cite[Theorem 26.1.1]{Hormander-4}) $\WF(f)$ is invariant by the symplectic lift $d\varphi_t^{-\top}$ of $(\varphi_t)_{t \in \R}$. But the maximal flow-invariant subset of $T^*(S\Sigma)$ contained in $\left\{ \xi(V(v)) \geq 0, \xi(X(v))=0 \right\}$ is $\mc{C}$, so this proves the claim. \textcolor{black}{By \cite[Proposition 11.3.3]{FriedlanderJoshi}, for any distribution $f\in \mc{D}'(S\Sigma)$ we have ${\rm WF}({\pi_n}_*f)=\emptyset$ if ${\rm WF}(f)\cap  ((E^0)^* \oplus \HH^*)=\emptyset$. Since  $\mc{C} \cap ((E^0)^* \oplus \HH^*) = \left\{0\right\}$, we deduce that ${\pi_n}_* f$ has empty wavefront set, i.e. ${\pi_n}_* f \in C^\infty(\Sigma,K^{\otimes n})$.}
\end{proof}

\subsubsection{Tensor tomography}

The tensor tomography problem consists in studying the transport equation $X u = f$, where $u, f \in C^\infty(S\Sigma)$ are smooth, $f$ has a finite Fourier expansion of degree $n$ and showing that $u$ has finite Fourier expansion of degree $n-1$ (with $u_0$ constant for $n=0$).

\begin{theorem}[Tensor tomography]
\label{theorem:tensor}
Assume that $(\Sigma, g)$ is Anosov. \textcolor{black}{If $X u = f$ with $f, u \in C^\infty(S\Sigma)$ and $f$ is such that $f_{k}=0$ for $|k|\geq n+1$, then $u_k=0$ for all $|k|\geq n$ (with $u_0$ constant if $n=0$).}

\end{theorem}

\textcolor{black}{The proof of Theorem \ref{theorem:main} only requires the cases $n=0,1$.}

Theorem \ref{theorem:tensor} was first proved in negative curvature in \cite{Guillemin-Kazhdan-80} and generalized to the Anosov setting in \cite{Dairbekov-Sharafutdinov-03} for $n=0,1$, \cite{Paternain-Salo-Uhlmann-14-2} for $n\leq 2$, and \cite{Guillarmou-17-1} for all $n\geq 2$. Note that Theorem \ref{theorem:tensor} is equivalent to the fact that the X-ray transform operator $I_n^g : C^\infty(\Sigma, S^n T^*\Sigma) \to \ell^\infty(\mc{C})$ defined by
\[
I^g_nh(c) := \dfrac{1}{\mc{L}_g(c)} \int_0^{\mc{L}_g(c)} h_{\gamma(t)}(\dot{\gamma}(t),...,\dot{\gamma}(t))~ \dd t
\]
is injective when restricted to divergence-free symmetric $n$-tensors, see \cite[Chapter 16]{Lefeuvre-book} for instance. For $n=2$, this operator is the linearization (at $g$) of the marked length spectrum operator $g \mapsto \mc{L}_g$.

\subsection{The pairing}

\label{ssection:pairing}

Given a closed geodesic $\gamma$ we set $S^1\gamma =\pi^{-1}(\gamma)\subset S\Sigma$ to be the circle bundle over $\gamma$. If $[0,\ell_g(\gamma)] \ni \tau \mapsto \gamma(\tau)$ is an arc-length parametrization of $\gamma$, we can parametrize $S^1\gamma$ by coordinates $(\tau,\theta) \in \R/(\ell_g(\gamma)\Z) \times \R/(2\pi\Z)$ such that:
\[
S^1\gamma = \left\{ (\gamma(\tau), R_\theta \dot{\gamma}(\tau)) ~|~ \tau \in [0,\ell_g(\gamma)], \theta \in [0,2\pi]\right\},\]
where $R_\theta$ denotes rotation by angle $\theta$ in $S\Sigma$.
Note that the tangent space of $S^1\gamma$ at the point $(\tau,\theta)$ is spanned by $((\cos \theta) X + (\sin \theta) H, V)$. We then endow $S^1\gamma$ with the orientation given by $(\cos \theta~ \lambda + \sin \theta~ \beta) \wedge \psi$. From now on, $S^1\gamma$ will always denote the \emph{oriented} submanifold of $S\Sigma$ with orientation given by the $2$-form above.
The following holds:

\begin{lemma}\label{formule_liouvillegeod}
In the coordinate system $(\tau,\theta)$, $d\lambda =\textcolor{black}{\sin \theta~ \dd\tau\wedge \dd\theta}$.
\end{lemma}

Note that for $\theta \in (0,\pi)$, this is a \textcolor{black}{\emph{positive}} multiple of the $2$-form defining the orientation on $S^1\gamma$.

\begin{proof}
This follows from \eqref{equation:dlambda}, see also \cite[p. 156]{Otal-90}.
\end{proof}

We next derive the following remarkable formula:

\begin{lemma}
\label{lemma:magical}
Let $f \in \mc{A}_{+}(S\Sigma)$ and $\gamma$ be a closed oriented geodesic. Then the pairing $(S^1\gamma, f \dd\lambda)$ is well-defined and
\begin{equation}
\label{equation:magical}
2i\int_{S^1\gamma} f~ \dd\lambda = \textcolor{black}{-}  \int_{[\gamma]} {\pi_{1}}_* f,
\end{equation}
where $[\gamma] \in H_1(\Sigma,\Z)$ denotes the homology class of the geodesic $\gamma$.
\end{lemma}

{\color{black}The distributional $2$-form $f \dd \lambda$ is defined as follows: given a $1$-form $\omega \in C^\infty(S\Sigma, T^*(S\Sigma))$, we can decompose $\omega = \omega_\lambda \lambda + \omega_\beta \beta + \omega_\psi \psi$, and then
\[
(f \dd \lambda, \omega) := (f, \omega_\lambda \dd \lambda \wedge \lambda).
\]}
Actually, if $f \in C^\infty(S\Sigma)$ (or if $f$ is merely a distribution but the pairing with $S^1\gamma$ makes sense), we will prove the formula
\begin{equation}
\label{equation:smooth}
\frac{i}{\pi} \int_{S^1\gamma} f~ \dd\lambda =\textcolor{black}{-}  \int_0^{\ell_g(\gamma)} f_1(\gamma(\tau),\dot{\gamma}(\tau)) \dd \tau \textcolor{black}{+}\int_0^{\ell_g(\gamma)} f_{-1}(\gamma(\tau),\dot{\gamma}(\tau)) \dd \tau,
\end{equation}
where $\tau \mapsto \gamma(\tau)$ is an arc-length parametrization of the geodesic. It then suffices to observe that \eqref{equation:magical} is a mere rewriting of \eqref{equation:smooth} since for $f \in \mc{A}_{+}(S\Sigma)$, ${\pi_{1}}_*f$ is a holomorphic differential (and the factor $2\pi$ comes from \eqref{equation:proj}).

\begin{remark}[Generalized intersection number]
\label{remark:number}
The pairing $(S^1\gamma, f\dd\lambda)$ can be thought of as a generalized \emph{intersection number}, following the terminology used by Bonahon \cite{Bonahon-88}, Otal \cite{Otal-90} and others. The difference is that we allow $f$ to be \emph{any} fiberwise holomorphic invariant distribution in $\mc{A}_{+}(S\Sigma)$ whereas, in some sense, the aforementioned authors only allowed $f$ to be a constant, which is a (very) particular case of a fiberwise holomorphic distribution. Moreover, we integrate over the full circle bundle $S^1\gamma$ whereas these authors integrate over $S^1\gamma_+$, that is, for angles $\theta \in (0,\pi)$. The usual intersection number of the \emph{Liouville current} $\eta$ (that is, the Liouville form modded out by the flow direction in the universal cover) with the geodesic $\gamma$ is commonly denoted by $\iota(\eta,\gamma)$ and is equal to $2\ell_g(\gamma)$ (or $\ell_g(\gamma)$, depending on the convention). It simply corresponds for us to $(S^1\gamma_+,d\lambda)$.\end{remark}

\begin{proof}[Proof of Lemma \ref{lemma:magical}]
By the wavefront set calculus (see \cite[Corollary 8.2.7]{Hormander-90} for instance), the pairing $(S^1\gamma,f \dd\lambda)$ is well-defined as long as
\begin{equation}
\label{equation:inter}
N^*(S^1\gamma) \cap \WF(f \dd\lambda) = \emptyset.
\end{equation}
Now, the conormal $N^*(S^1\gamma)$ to $S^1\gamma$ is a line contained in $(E^0)^* \oplus \HH^*$ and by Lemma \ref{lemma:wf-c} and \eqref{equation:disjoint}, the intersection \eqref{equation:inter} is indeed empty. Hence, the pairing $(S^1\gamma,f \dd\lambda)$ is well-defined and extends the pairing computed for $f \in C^\infty(S\Sigma)$. As a consequence, it suffices to establish \eqref{equation:smooth} for smooth functions and \eqref{equation:magical} then follows immediately.

Now, if $f \in C^\infty(S\Sigma)$, we can decompose $f$ in Fourier modes along $S^1\gamma$ and write
\begin{equation}
\label{equation:decomp2}
f(\tau,\theta) = \sum_{k \in \Z} a_k(\tau) e^{ik \theta}.
\end{equation}
\textcolor{black}{Then, using \eqref{equation:decomp2}, Lemma \ref{formule_liouvillegeod}} and $\sin \theta = (2i)^{-1}(e^{i \theta} - e^{-i \theta})$, we get:
\[
\begin{split}
\frac{i}{\pi}  \int_{S^1\gamma} f~ d\lambda &= \textcolor{black}{+}\frac{i}{\pi}  \int_0^{\ell_g(\gamma)} \int_0^{2\pi} f(\tau,\theta) \sin \theta~ \dd \theta \dd \tau  \\
& = \textcolor{black}{-}\int_0^{\ell_g(\gamma)} a_1(\tau) \dd \tau \textcolor{black}{+}\int_0^{\ell_g(\gamma)} a_{-1}(\tau) \dd \tau.
\end{split}
\]
The previous equality corresponds exactly to \eqref{equation:smooth}.
\end{proof}

\section{Proof of Theorem \ref{theorem:main}}

\label{section:proof}

In this section, we first prove that $\mc{L}_g$ determines the class $[J_g] \in \mc{M}(\Sigma)$ in the moduli space. Note that all the objects defined above (complex structures, stable and unstable bundles, etc.) now depend on the metric $g_1$ or $g_2$ and we shall add a subscript to distinguish them. The following is key to Theorem \ref{theorem:main}:

\begin{proposition}
\label{proposition:key}
Let $g_1,g_2$ be two smooth Anosov metrics on $\Sigma$. If $\mc{L}_{g_1} = \mc{L}_{g_2}$ then $[J_1]=  [J_2]$ in $\mc{M}(\Sigma)$. Equivalently, there exists a diffeomorphism $\psi : \Sigma \to \Sigma$ such that $\psi^*J_2 = J_1$ and $\psi^* g_2 = e^{2 f} g_1$ for some $f \in C^\infty(\Sigma)$.
\end{proposition}

This should be compared with the case of a surface with boundary, see \cite{Pestov-Uhlmann-05,Guillarmou-17-2}. The proof of Proposition \ref{proposition:key} is the content of \S\ref{ssection:algebra} and \S\ref{ssection:period}. We then prove in Proposition \ref{proposition:key2} that we can recover from $\mc{L}_g$ the class $[J_g] \in \mc{T}(\Sigma)$ in Teichmüller space, and then the conformal factor of the metric (the latter is a standard argument due to Katok \cite{Katok-88}).

\subsection{The isomorphism}

\label{ssection:algebra}

There is a natural identification of $S\Sigma_1$ with $S\Sigma_2$ by simply scaling the fibers via the map
\[
s : S\Sigma_1 \longrightarrow S\Sigma_2, \qquad v \mapsto v/|v|_{g_2}.
\]
The following holds:

\begin{proposition}
\label{proposition:algebra-iso}
Assume that $\mc{L}_{g_1}=\mc{L}_{g_2}$. Then there exists a smooth diffeomorphism $\phi : S\Sigma_2 \to S\Sigma_1$ such that:
\begin{enumerate}[label=\emph{(\roman*)}]
\item $\phi \circ \varphi_t^{g_2} = \varphi_t^{g_1} \circ \phi$, for all $t \in \R$,
\item $s \circ \phi : S\Sigma_2 \to S\Sigma_2$ is isotopic to the identity,
\item $\phi$ induces a linear isomorphism:
\[
\phi^* : \mathcal A_+(S\Sigma_1) \longrightarrow \mathcal A_+(S\Sigma_2),
\]
by pullback of distributions.
\end{enumerate}
\end{proposition}



Proposition \ref{proposition:algebra-iso} will rely on the following lemma:

{\color{black}\begin{lemma}
\label{lemma:gabriel}
Let $\phi:S\Sigma_{2}\to S\Sigma_{1}$ be a diffeomorphism satisfying \emph{(i)} and \emph{(ii)} in Proposition \ref{proposition:algebra-iso}.
Then $\phi$ preserves the natural orientation of the weak unstable bundle, namely that given by the basis
$\{X,Y^{u}=H+r_{+}V\}$.
\end{lemma}

\begin{proof} For this proof it is convenient to pass to the universal cover $\widetilde{\Sigma}$ of $\Sigma$. We lift the weak unstable foliations
to $S\widetilde{\Sigma}_{2}$ and $S\widetilde{\Sigma}_{1}$ and the diffeomorphism $\phi$ to a diffeomorphism $\tilde{\phi}$.
Observe that $\phi$ preserves the orientation of the weak unstable bundles if and only if $\tilde{\phi}$ does.

Now we use that the weak unstable leaves are graphs and hence the weak unstable foliations are conjugate to the product foliation in $\widetilde{\Sigma}\times S^{1}$.
Hence we may pick {\it closed} 1-forms $\omega_{i}$ such that $\omega_{i}(V_{i})>0$ and $\omega_{i}\in C^{1}(S\widetilde{\Sigma}_{i}, (E^u_{i})^*)$ for $i=1,2$.
Since $\tilde{\phi}$ conjugates the weak foliations there is a never vanishing function $f \in C^1(S\widetilde{\Sigma}_{2})$ such that
\[\tilde{\phi}^*\omega_{1}=f\,\omega_{2}.\]
Our claim is reduced to showing that $f>0$. Fix a point $x_0\in \widetilde{\Sigma}$ and consider the vertical fibres $S^{1}_{i} \subset S\widetilde{\Sigma}_i$ over $x_{0}$ with their canonical orientations (i.e. that induced by $V_{i}$) for $i=1,2$.  Then we have
\begin{equation}
\int_{\tilde{\phi}^{-1}(S^{1}_{1})}\tilde{\phi}^*\omega_{1}=\int_{S^{1}_{1}}\omega_{1}>0.
\label{eq:plus}
\end{equation}
Since $s\circ\phi$ is isotopic to the identity, $\tilde{\phi}^{-1}(S^{1}_{1})$ is homotopic to $S^{1}_{2}$ and therefore using that $\tilde{\phi}^*\omega_{1}$ is closed we derive
\[\int_{\tilde{\phi}^{-1}(S^{1}_{1})}\tilde{\phi}^*\omega_{1}=\int_{S^{1}_{2}}\tilde{\phi}^*\omega_{1}.\]
Using \eqref{eq:plus} yields
\[\int_{S^{1}_{2}}f\,\omega_{2}>0\]
and thus $f>0$ as desired.
\end{proof}}

We now prove Proposition \ref{proposition:algebra-iso}.

\begin{proof}[Proof of Proposition \ref{proposition:algebra-iso}] Two Anosov geodesic flows are orbit equivalent via a H\"older homeomorphism isotopic to the identity (see \cite{Ghys-84,Gromov-00} and \cite[Lemma B.1]{Guillarmou-Knieper-Lefeuvre-22}).  The equality of the marked length spectra $\mc{L}_{g_1}=\mc{L}_{g_2}$ 
and the Liv\v sic theorem imply that such orbit equivalence can be upgraded to a conjugacy $\phi$.
Moreover, such a conjugacy is necessarily smooth for three-dimensional contact Anosov flows: the $C^1$-regularity follows from \cite{Feldman-Ornstein-87} and the $C^\infty$ regularity from \cite{Delallave-Moriyon88} (see also \cite{Gogolev-RodriguezHertz-22} for a more general smoothness result on conjugacies of three-dimensional volume preserving Anosov flows).

As $\phi$ maps $E^s_{2}$ (resp. $E^u_2$) to $E^s_1$ (resp. $E^u_1$), we get that $\ker \lambda_2 = E^s_2 \oplus E^u_2$ is mapped by $\phi$ to $\ker \lambda_1$. Hence, $\phi^* \lambda_1 = \lambda_2$ and $\phi^*(-\lambda_1 \wedge \dd \lambda_1) = -\lambda_2 \wedge \dd \lambda_2$, that is, $\phi$ preserves the Liouville volume form hence the orientation.

As $\phi$ preserves the orientation, $\Sigma$ is connected, and $d \phi^{\top}(\lambda_1) = \lambda_2$, we deduce that $\mc{C}' := d\phi^{\top}(\mc{C}_{1}) = \pm \mc{C}_{2}$. Indeed, since $d\phi^\top$ maps connected sets to connected sets, $(E_1^u)^*$ to $(E^u_2)^*$ and $(E^s_1)^*$ to $(E^s_2)^*$, $d\phi^\top(\mc{C}_1)$ must be one of the 4 cones in the right Figure \ref{figure} (inside the subspace $\{\xi(X_2(v))=0\}$). Now, if $\mc{C}_{1}$ were mapped by $\phi$ to one of the two cones on Figure \ref{figure} other than $\pm \mc{C}_{2}$, $\phi$ would reverse the orientation, which is absurd.

{\color{black}We now further claim that $\mc{C}' = + \mc{C}_2$. Indeed, if $\mc{C}' = -\mc{C}_2$ then
$\phi$ flips the orientation of weak unstable leaves which is a contradiction by Lemma \ref{lemma:gabriel}. Hence $\mc{C}' = + \mc{C}_2$.}

Finally, we show that $\phi^*$ maps $\mc{A}_+(S\Sigma_1)$ to $\mc{A}_+(S\Sigma_2)$. Let $f \in \mathcal{A}_+(S\Sigma_1)$ and define $f' := \phi^*f \in \mc{D}'(S\Sigma_2)$. By Lemma \ref{lemma:wf-c}, the wavefront set of $f$ is contained in $\mc{C}_{1}$ so the wavefront set of $f'$ is contained in the cone $\mc{C}' = d\phi^{\top}(\mc{C}_{1}) = \mc{C}_2$. Then $\mathbf{S}_- f' \in C^\infty(S\Sigma_2)$ and, since $X_2f'=\phi^*(X_1f)=0$, we also have that $X_{2}(\mathbf{S}_- f') = \eta_+ f'_0 + \eta_+ f'_{-1}$. By the tensor tomography Theorem \ref{theorem:tensor}, we deduce that $\mathbf{S}_- f'=f_0'$  and thus  $f'$ is fiberwise holomorphic. This proves the claim. That $\phi^*:\mathcal A_+(S\Sigma_1) \longrightarrow \mathcal A_+(S\Sigma_2)$ is an isomorphism now follows from the fact that it admits an inverse $(\phi^{-1})^*$.
\end{proof}

\begin{remark} 
\textcolor{black}{The map $\phi^*:\mathcal A_+(S\Sigma_1) \longrightarrow \mathcal A_+(S\Sigma_2)$ is an algebra isomorphism as the product property is obviously satisfied.}
\end{remark}

\subsection{Extension operator. Period preservation}

\label{ssection:period}

The aim of this paragraph is to show that, in some sense, the isomorphism $\phi^*$ maps holomorphic differentials of the first surface to the second, and that it preserves periods.

We first claim that the map
\begin{equation}
\label{equation:surj-pi1}
{\pi_1}_* : \mc{A}_{{\color{black}+}}(S\Sigma)  \longrightarrow H^{0}_J(\Sigma, K)
\end{equation}
is well-defined. Indeed, a fiberwise holomorphic invariant distribution $f \in \mc{A}_{{\color{black}+}}(S\Sigma)$ satisfies $Xf = 0$ so its first Fourier mode $f_1$ satisfies $\eta_- f_1 = 0$ by Lemma \ref{lem:A+_holof_0f_1} which, in turn, is equivalent to $\bar{\partial} {\pi_1}_* f = 0$.

A key input that we now use is that the map \eqref{equation:surj-pi1} is known to be surjective by \cite[Theorem 1.5]{Paternain-Salo-Uhlmann-14-2}. This allows us to define a right-inverse $\mathbf{e}_1$ such that
\[\e_1 : H^0_J(\Sigma,K) \longrightarrow \mc{A}_{{\color{black}+}}(S\Sigma) , \quad {\pi_1}_* \circ \e_1 = \mathbf{1}_{H^0_J(\Sigma,K)},\]
which we call the \emph{extension operator}.

\begin{remark}
More generally, using the spectral theory of Anosov flows, it was proved in \cite[Corollary 3.8]{Guillarmou-17-1} that the map ${\pi_n}_* : \mc{A}_{+}(S\Sigma)\cap \bigcap_{m<n} \ker {\pi_m}_* \longrightarrow H^0_J(\Sigma, K^n)$ is surjective for all $n \geq 0$ (along with other more general results on the existence of flow-invariant distributions with prescribed $n$-th Fourier mode). {\color{black} We also observe that $\mathbf{e}_1$ is not uniquely defined since $\ker {\pi_1}_*\cap \mc{A}_+(S\Sigma)\not=0$.}
\end{remark}

Let 
\[
F : H^0_{J_1}(\Sigma, K) \longrightarrow H^0_{J_2}(\Sigma, K)
\]
be the map defined by the following commutative diagram:
\[
\xymatrix{
    \mc{A}_{{\color{black}+}}(S\Sigma_{1})  \ar[r]^{\phi^*}  & \mc{A}_{{\color{black}+}}(S\Sigma_{2}) \ar[d]_{{\pi_1}_*} \\
    H^0_{J_1}(\Sigma, K) \ar[u]_{\e_1} \ar[r]^F & H^0_{J_2}(\Sigma, K)
  }
\]
The following result unlocks Theorem \ref{theorem:main}:

\begin{proposition}
\label{proposition:period}
The map
\[
F : H^0_{J_1}(\Sigma, K) \longrightarrow H^0_{J_2}(\Sigma, K)
\]
 is a period-preserving $\C$-linear isomorphism in the following sense: for all $[\gamma] \in H_1(\Sigma,\Z)$, $\omega \in  H^0_{J_1}(\Sigma, K) $,
 \begin{equation}
 \label{equation:period}
 \int_{[\gamma]} \omega = \int_{[\gamma]} F \omega.
 \end{equation}
\end{proposition}

That $F$ is an isomorphism follows immediately from the period preservation \eqref{equation:period}. Hence, the only non-trivial part of Proposition \ref{proposition:period} is \eqref{equation:period}. Here and below, given a homology class $[\gamma] \in H_1(\Sigma,\Z)$, we shall denote by $\gamma_{g_i}$ a closed oriented geodesic (for the metric $g_i, i=1,2$) representing $[\gamma]$, and $\gamma$, when no preferred metric is chosen.
If $\omega \in H^0_J(\Sigma,K)$, $[\gamma] \in H_1(\Sigma,\Z)$, Lemma \ref{lemma:magical} yields
\begin{equation}
\label{equation:integration}
2 i  \int_{S^1\gamma} \mathbf{e}_1 \omega~\dd\lambda = \textcolor{black}{-}\int_{[\gamma]} \omega.
\end{equation}
We can now complete the proof of Proposition \ref{proposition:period}.

\begin{proof}[Proof of Proposition \ref{proposition:period}]
Let $[\gamma] \in H_1(\Sigma,\Z)$ and let $\gamma_{g_1}(c), \gamma_{g_2}(c)$ be two geodesic representatives for $[\gamma]$ (with respect to $g_1$ and $g_2$) in the same free homotopy class $c \in \mc{C}$. 
\textcolor{black}{Note that $[s^{-1}(S^1\gamma_{g_2}(c))]=[S^1\gamma_{g_1}(c)]$ in $H_2(S\Sigma_1,\Z)$ and since
$s\circ \phi$ is isotopic to the identity, we get  $[\phi(S^1\gamma_{g_2}(c))] = [s^{-1}(S^1\gamma_{g_2}(c))]=[S^1\gamma_{g_1}(c)]$ in $H_2(S\Sigma_1,\Z)$.}

\textcolor{black}{Notice that the pairing $(\phi(S^1\gamma_{g_2}(c)), \mathbf{e}_1 \omega \dd\lambda_1)$ is well-defined, similarly to Lemma \ref{lemma:magical}: indeed, $N^*\phi(S^1\gamma_{g_2}(c))\cap {\rm WF}( \mathbf{e}_1 \omega \dd\lambda_1)=\emptyset$ since $\mc{C}_2\cap N^*S^1\gamma_{g_2}(c)=\emptyset$ and $\phi^*( \mathbf{e}_1 \omega~\dd\lambda_1)\subset d\phi^{\top}(\mc{C}_1)\subset \mc{C}_2$ by the proof of Proposition \ref{proposition:algebra-iso}.}

Moreover, given a distribution $f \in \mc{A}_+(S\Sigma_1)$, observe that $f \dd\lambda$ is a closed $2$-form. We claim that
\[
\int_{S^1\gamma_{g_1}(c)} \mathbf{e}_1 \omega~\dd\lambda_1= \int_{\phi(S^1\gamma_{g_2}(c))} \mathbf{e}_1 \omega~\dd\lambda_1.
\]
\textcolor{black}{Indeed, by \cite[Lemma 2.1]{Dyatlov-Zworski-17},  there is a smooth closed $2$-form $\alpha$ such that $f\dd\lambda =\alpha +\dd u$ where ${\rm WF}(u)\subset {\rm WF}(f)$.}
Hence, both pairings $(S^1\gamma_{g_1}(c), \dd u)$ and $(\phi(S^1\gamma_{g_2}(c)), \dd u)$ are well-defined by the wavefront set calculus (same arguments as for $f\dd\lambda$) and equal to $0$ since $\dd u$ is exact.

Then, we get:
\[
\begin{split}
\int_{S^1\gamma_{g_1}(c)} \mathbf{e}_1 \omega~\dd\lambda_1 & = \int_{S^1\gamma_{g_1}(c)} \alpha =  \int_{\phi(S^1\gamma_{g_2}(c))} \alpha = \int_{\phi(S^1\gamma_{g_2}(c))} \mathbf{e}_1 \omega~\dd\lambda_1,
\end{split}
\]
where the second equality simply follows from the fact that $\alpha$ is \textcolor{black}{closed} and $[S^1\gamma_{g_1}(c)] = [\phi(S^1\gamma_{g_2}(c))]$ in $H_2(S\Sigma_1,\Z)$.

As a consequence, given $\omega \in H^0_{J_1}(\Sigma_1,K)$ we get by \eqref{equation:integration} and the discussion above that
\[
\begin{split}
\int_{[\gamma]} \omega &= \textcolor{black}{-}2i \int_{S^1\gamma_{g_1}(c)} \mathbf{e}_1 \omega~\dd\lambda_1 \\
&  =\textcolor{black}{-} 2i \int_{\phi(S^1\gamma_{g_2}(c))} \mathbf{e}_1 \omega~\dd\lambda_1 = \textcolor{black}{-}2i\int_{S^1\gamma_{g_2}(c)} \phi^*(\mathbf{e}_1 \omega~\dd\lambda_1) \\
& =\textcolor{black}{-} 2i \int_{S^1\gamma_{g_2}(c)} \phi^*(\mathbf{e}_1\omega) ~\dd \lambda_2 = \int_{[\gamma]} {\pi_1}_*  \phi^*(\mathbf{e}_1\omega) = \int_{[\gamma]} F\omega.
\end{split}
\]
This concludes the proof.
\end{proof}

We can now conclude the proof of Proposition \ref{proposition:key}.

\begin{proof}[Proof of Proposition \ref{proposition:key}]
If $\mc{L}_{g_1} = \mc{L}_{g_2}$, we get by the above results that
\[
F : H^0_{J_1}(\Sigma,K) \longrightarrow H^0_{J_2}(\Sigma,K)
\]
is a period-preserving $\C$-linear complex isomorphism. {\color{black}We claim that $(\Sigma,J_1)$ and $(\Sigma,J_2)$ have same period matrix. Indeed, let $\left\{\zeta_i\right\}$ be a basis of the space $H^0_{J_1}(\Sigma,K)$ such that \eqref{equation:delta} is satisfied. Then $\left\{F(\zeta_i)\right\}$ is a basis of $H^0_{J_2}(\Sigma,K)$ such that \eqref{equation:delta} is also satisfied. Moreover,
\[
\Pi(J_1)_{jk} = \int_{b_j} \zeta_{k} = \int_{b_j} F(\zeta_k) = \Pi(J_2)_{jk}.
\] 
By the Torelli Theorem \ref{theorem:torelli}, it follows that $(\Sigma,J_1)$ and $(\Sigma,J_2)$ are biholomorphic and differ in the Teichm\"uller space by an element $[\psi] \in \mathrm{MCG}(\Sigma)$.} This is the content of Proposition \ref{proposition:key}.
\end{proof}

\subsection{End of the proof}

We now prove that $\mc{L}_g$ determines the class of $[J_g] \in \mc{T}(\Sigma)$ in the Teichmüller space of $\Sigma$.

\begin{proposition}
\label{proposition:key2}
Let $g_1,g_2$ be two smooth Anosov metrics. If $\mc{L}_{g_1} = \mc{L}_{g_2}$ then $[J_{g_1}]=[J_{g_2}]$ in $\mc{T}(\Sigma)$. Equivalently, there exists a diffeomorphism $\psi : \Sigma \to \Sigma$ isotopic to the identity such that $\psi^* g_2 = e^{2 f} g_1$ for some $f \in C^\infty(\Sigma)$.
\end{proposition}

Proposition \ref{proposition:key2} will follow from Proposition \ref{proposition:key} and the following:

\begin{lemma}
\label{lemma:cover}
Let $J_1$ and $J_2$ be two complex structures on $\Sigma$ compatible with orientation such that $[J_1] \neq [J_2]$ in $\mc{T}(\Sigma)$. Then, there exists a finite cover $\Sigma' \to \Sigma$ such that the lifts $[J_1'], [J_2'] \in \mc{T}(\Sigma')$ are not in the same $\mathrm{MCG}(\Sigma')$-orbit.
\end{lemma}

Let us first derive Proposition \ref{proposition:key2} from this.

\begin{proof}[Proof of Proposition \ref{proposition:key2}]
If $\mc{L}_{g_1} = \mc{L}_{g_2}$ on $\Sigma$, then the same equality $\mc{L}_{g'_1} = \mc{L}_{g'_2}$ holds on any finite cover $\Sigma' \to \Sigma$. Thus, the desired conclusion follows by combining directly Proposition \ref{proposition:key} with Lemma \ref{lemma:cover}.
\end{proof}

The proof of Lemma \ref{lemma:cover} was kindly explained to us by M. Wolff.

\begin{proof}[Proof of Lemma \ref{lemma:cover}]
In order to simplify notation, given a free homotopy class $c\in\mc{C}$ we write $\ell_g(c)$ for the length of the unique $g$-geodesic in the free homotopy class $c$. We use the identification of Teichmüller space $\mc{T}(\Sigma)$ with marked hyperbolic structures on $\Sigma$. Let $[h_1], [h_2] \in \mc{T}(\Sigma)$ be the (distinct) hyperbolic structures given by $J_1$ and $J_2$. Let $\alpha \in \mc{C}(\Sigma)$ be a free homotopy class of simple closed curves on $\Sigma$ such that $\ell_{h_1}(\alpha) \neq \ell_{h_2}(\alpha)$. Note that such a class always exists otherwise we would have $[h_1]=[h_2]$ by the $9g-9$ Theorem (see \cite[Theorem 10.7]{Farb-Margalit-11}).

We then consider a finite cover $\Sigma' \to \Sigma$ which unfolds every simple closed curve, except $\alpha$. More precisely, for $K > 0$ fixed, we consider a finite cover $\pi : \Sigma' \to \Sigma$ such that the following holds: denoting by $h_1', h_2'$ the lifts of $h_1,h_2$, there exists a free homotopy class of simple closed curves $\alpha' \in \mc{C}(\Sigma')$ on $\Sigma'$ such that $\pi : \gamma_{h'_1}(\alpha') \to \gamma_{h_1}(\alpha)$ is $1$-to-$1$, and for every free homotopy class of simple closed curves $\beta \in \mc{C}(\Sigma')$ such that $\beta \neq \alpha'$ , $\ell_{h'_1}(\beta) > K$ (see \cite[Appendix A, Theorem C]{Alvarez-Brum-Martinez-Potrie-19} for the existence of such a covering). In particular, choosing $K > 0$ such that $K > \ell_{h_1}(\alpha)$, we can guarantee that $\gamma_{h'_1}(\alpha')$ is the systole of $(\Sigma',h_1')$ with length $\ell_{h_1}(\alpha)$.

Now, as $\Sigma$ is compact, there exists a constant $C > 1$ such that for all $c \in \mc{C}$ (not necessarily simple),
\[
1/C \leq \dfrac{\ell_{h_1}(c)}{\ell_{h_2}(c)} \leq C.
\]
We fix $K > 0$ in the construction above such that $K/C > \max(\ell_{h_1}(\alpha),\ell_{h_2}(\alpha))$. We claim that $\gamma_{h'_2}(\alpha')$ is also the systole of $(\Sigma',h_2')$ with length $\ell_{h_2}(\alpha)$. Indeed, observe first that $\pi : \gamma_{h'_2}(\alpha') \to \gamma_{h_2}(\alpha)$ is also $1$-to-$1$ (so $\ell_{h_2'}(\alpha') = \ell_{h_2}(\alpha)$), and for every free homotopy class of simple closed curves $\beta \in \mc{C}(\Sigma'), \beta \neq \alpha'$, one has:
\[
\begin{split}
\ell_{h_2'}(\beta) & =\ell_{h_1'}(\beta) \cdot \dfrac{\ell_{h_2'}(\beta)}{\ell_{h_1'}(\beta)} \\
& = \ell_{h_1'}(\beta) \cdot \dfrac{\ell_{h_2}(\pi(\beta))}{\ell_{h_1}(\pi(\beta))} > K/C > \max(\ell_{h_1}(\alpha),\ell_{h_2}(\alpha)) \geq \ell_{h_2'}(\alpha').
\end{split}
\]
As a consequence, the systoles of $(\Sigma',h_1')$ and $(\Sigma',h_2')$ have different lengths, so $[h_1']$ and $[h_2']$ are not in the same $\mathrm{MCG}(\Sigma')$-orbit.
\end{proof}

Finally, we complete the proof of Theorem \ref{theorem:main}.

\begin{proof}[Proof of Theorem \ref{theorem:main}]
If $\mc{L}_{g_1}=\mc{L}_{g_2}$, Proposition \ref{proposition:key2} yields that there exists $\psi : \Sigma \to \Sigma$ isotopic to the identity such that $\psi^*g_1 = e^{2f} g_2$. But then, applying Katok's argument \cite{Katok-88} for conformal metrics with same marked length spectrum, we get that $f=0$. This concludes the proof.
\end{proof}

We conclude with a remark on Theorem \ref{theorem:main} for metrics with low regularity:

\begin{remark}
\label{remark:regularity}Theorem \ref{theorem:main} still holds for metrics with $C^4$-regularity. Indeed, if $g_1$ and $g_2$ are $C^k$-regular, the unit tangent bundles $S\Sigma_{1},S\Sigma_2$ are also $C^k$-regular and the geodesic vector fields $X_1,X_2$ are $C^{k-1}$. The conjugacy $\phi$ in Proposition \ref{proposition:algebra-iso} is then $C^{k-1}$ by \cite[Theorem 1.1, item 1]{Gogolev-RodriguezHertz-22}. By \cite[Theorem 5.5]{Paternain-Salo-Uhlmann-14-2}, the right-inverse $\e_1$ defined in \S\ref{ssection:period} can be constructed with values in $H^{-1}(S\Sigma)$ flow-invariant distributions, and the proof of \cite[Theorem 5.5]{Paternain-Salo-Uhlmann-14-2}, based on the Pestov identity, applies with metrics which are only $C^2$-regular. 

In order to prove that $\phi^*$ maps $\mc{A}_+(S\Sigma_1)$ to $\mc{A}_+(S\Sigma_2)$ in Proposition \ref{proposition:algebra-iso}, one can use the notion of $H^s$ 
wave-front set ${\rm WF}_s$ of distributions (for $s\geq -1$) in the Sobolev space $H^{-1}(S\Sigma_{1,2})$, see \cite[Section 8.2]{Hormander-nonlinear} for a definition. Indeed, the action by pullback by $\phi^*$ of  distributions in $H^{-1}(S\Sigma_1)$ is well-defined and maps to distributions with regularity $H^{-1}(S\Sigma_2)$  if $\phi$ is at least $C^2$, using that $C^1$ embeds into $H^1$, the dual of $H^{-1}$. For $u \in H^{-1}(S\Sigma_1)$, the pull-back $\phi^*u \in H^{-1}(S\Sigma_2)$ is such that $\dd\phi^{\top}$ sends ${\rm WF}_s(u)$ to ${\rm WF}_s(\phi^*u)$ if $\phi$ is $C^{k-1}$ and $-1 \leq s \leq k-2$. The tensor tomography result of Theorem \ref{theorem:tensor} applies to functions and $1$-forms in $H^2$ (and for $C^2$-regular metrics in order to apply \cite[Theorem 5.5]{Paternain-Salo-Uhlmann-14-2} based on the Pestov identity). Hence, we need to work with the ${\rm WF}_s(u)$ wave-front set with $s\geq 2$ and this forces $k\geq s+2\geq 4$.

\textcolor{black}{Finally, in the pairing formula of Lemma \ref{lemma:magical}, $S^1\gamma$ is $C^{k-1}$ (since $\gamma \subset S\Sigma$ is itself $C^{k-1}$ as an integral curve of the $C^{k-1}$-vector field $X$), and the restriction of an $H^{-1}(S\Sigma)$ distribution to $S^1\gamma$ is $H^{-1-1/2}$-regular (see Lemma \ref{reg} below); moreover, the contact form $\lambda$ is $C^{k-1}$, thus $\dd\lambda$ is $C^{k-2}$. The pairing $(S^1\gamma,f\dd\lambda)$ is then well-defined as long as $k-2 > 1+1/2$, that is $k \geq 4$. The end of the proof goes through as in the smooth case. Overall, one can run the arguments with $C^4$-regular metrics.}
\end{remark}

\begin{lemma}\label{reg}
Let $M$ be a closed manifold. If $u\in H^s(M)$ for $s\in \R$, $S$ is a smooth hypersurface and ${\rm WF}(u)\cap N^*S=\emptyset$, then 
the restriction $u|_{S}$ makes sense and belongs to $H^{s-1/2}(S)$.
\end{lemma}
\begin{proof}
The restriction makes sense by the wavefront set condition.
Using charts, it suffices to assume that $u\in H^s(\R^n)$ has compact support and $S=\{x_1=0\}$. Let $A > 0$. By the wavefront set assumption, for all $N >0$, there exists $C_N >0$ such that
$|\hat{u}(\xi_1,\xi')|\leq C_N\langle \xi\rangle^{-N}$ for all $|\xi_1|>A|\xi'|$, where $\xi=(\xi_1,\xi')$.
Set $f := u|_{S}$. We can thus write $2\pi\hat{f}(\xi')=\int_{\R} \hat{u}(\xi_1,\xi')d\xi_1$. Thus, for all $N>0$, 
there is $C_N' >0$ (depending linearly on $u$) such that
\[\begin{split}
2\pi|\hat{f}(\xi')| \leq & \int_{|\xi_1|\leq A|\xi'|} |\hat{u}(\xi)|d\xi_1+  \int_{|\xi_1|> A|\xi'|} C_N\langle \xi\rangle^{-N} d\xi_1\\
\leq &\Big(\int_{|\xi_1|\leq A|\xi'|} |\hat{u}(\xi)|^2d\xi_1\Big)^{1/2}(2A |\xi'|)^{1/2}+C'_N \langle \xi'\rangle^{-N+1}.
\end{split}\]
Therefore, taking $N$ large enough, one has for some $C_N''>0$ (depending linearly on $u$)
\[
\|f\|^2_{H^{s-1/2}} = \int_{\R^n} |\hat{f}(\xi')|^2\langle \xi'\rangle^{2s-1}d\xi'\leq   2A\int_{\R^n} |\hat{u}(\xi)|^2 \langle \xi'\rangle^{2s} d\xi_1d\xi'+C_N'' < \infty.
\]
This proves the claim.
\end{proof}

\bibliographystyle{alpha}
\bibliography{Biblio}

\end{document}